\documentclass[12pt,reqno]{amsart}
\usepackage{amsopn,amssymb,a4wide,amsmath}
\usepackage[square, numbers, comma, sort&compress]{natbib}

\newtheorem{thm}{Theorem}[section]
\newtheorem{lem}[thm]{Lemma}

\theoremstyle{definition}
\newtheorem{defn}[thm]{Definition}

\newcommand{\cl}[1]{\ensuremath{\overline{{#1}}}}
\newcommand{\ep}{\varepsilon}
\newcommand{\N}{\mathbb{N}}
\newcommand{\n}[1]{\ensuremath{\left\|{#1}\right\|}}
\newcommand{\ndot}{\ensuremath{\left\|\cdot\right\|}}
\newcommand{\R}{\mathbb{R}}
\newcommand{\restrict}[1]{\ensuremath{\!\!\upharpoonright_{#1}}}
\newcommand{\set}[2]{\ensuremath{\left\{{#1}\;:\;\,{#2}\right\}}}
\newcommand{\tn}[1]{\ensuremath{\tri{#1}\tri}}
\newcommand{\tndot}{\ensuremath{\tri\cdot\tri}}
\newcommand{\tri}{{\displaystyle |\kern-.9pt|\kern-.9pt|}}

\newcommand{\ttri}{|\kern-.9pt|\kern-.9pt|}
\newcommand{\ttrin}{\ttri\cdot\ttri}

\DeclareMathOperator{\aspan}{span}
\DeclareMathOperator{\ext}{ext}
\DeclareMathOperator{\sgn}{sgn}
\DeclareMathOperator{\strg}{str}
\DeclareMathOperator{\supp}{supp}

\allowdisplaybreaks

\begin{document}
\title{Smooth and polyhedral approximation in Banach spaces}
\date{\today}
\author{Victor Bible and Richard J.\ Smith}
\address{School of Mathematics and Statistics, University College Dublin, Belfield, Dublin 4, Ireland.}
\thanks{The authors are supported
financially by Science Foundation Ireland under Grant Number `SFI
11/RFP.1/MTH/3112'.}
%; secondary 46B26}

\begin{abstract}
We show that norms on certain Banach spaces $X$ can be approximated 
uniformly, and with arbitrary precision, on bounded subsets of $X$ by $C^\infty$ smooth norms and polyhedral norms. In particular, we show that this holds for any equivalent norm on $c_0(\Gamma)$, where $\Gamma$ is an arbitrary set. We also give a necessary condition for the existence of a polyhedral norm on a weakly compactly generated Banach space, which extends a well-known result of Fonf.
\end{abstract}
\maketitle

\section{Introduction}
Given a Banach space $(X,\ndot)$ and $\ep > 0$, we say that a new norm $\tndot$ is \emph{$\ep$-equivalent  to $\ndot$} if
\[
\tn{x} \leq \n{x} \leq (1 + \ep) \tn{x},
\]
for all $x \in X$. Suppose that P is some geometric property of norms, such as smoothness or strict convexity. We shall say that a norm $\ndot$ can be \emph{approximated by norms having P} if, given any $\ep>0$, there exists a norm having P that is \emph{$\ep$-equivalent to $\ndot$}. This is equivalent to the statement, often seen in the relevant literature, that $\ndot$ may be approximated uniformly, and with arbitrary precision, on bounded subsets of $X$ by norms having P.

The question of whether all equivalent norms on a given Banach space can be approximated by norms having P is a recurring theme in renorming theory. It is known to be true if P is the property of being strictly convex, or locally uniformly rotund (see \cite[Section II.4]{dgz:93}). (In fact, in these two cases, it is possible to show that if $\ndot$ has P, then the set of equivalent norms on $X$ having P is residual in the space of all equivalent norms on $X$, which is completely metrisable).

Several works in the literature, such as \cite{dfh:96,dfh:98,fhz:97,hj:14,hp:14,ht:14}, have addressed this question (or closely related questions) in the case of $C^k$ smoothness or polyhedrality; this question is again the subject of Section \ref{approx-norms} of this paper. 

\begin{defn}
We say the norm $\ndot$ of a Banach space $X$ is $C^k$ \emph{smooth} if its $k$th Fr\'{e}chet derivative exists and is continuous at every point of $X \setminus \{0\}$. The norm said to be \emph{$C^{\infty}$ smooth} if this holds for all $k \in \mathbb{N}$. 
\end{defn}

For separable spaces, we have the following recent and conclusive result.

\begin{thm}[{\cite[Theorem 2.10]{ht:14}}]\label{thm-ht}
Let $X$ be a separable Banach space with a $C^k$ smooth norm. Then any equivalent norm on $X$ can be approximated by $C^k$ smooth norms.
\end{thm}

There is an analogous result to Theorem \ref{thm-ht} for polyhedral norms.

\begin{defn}[{\cite[p.\ 265]{k:60}}]\label{polyhedral}
We say a norm $\ndot$ on a Banach space $X$ is \emph{polyhedral} if, given any finite-dimensional
subspace $Y$ of $X$, the restriction of the unit ball of $\ndot$ to $Y$ is a polytope.
\end{defn}

\begin{thm}[{\cite[Theorem 1.1]{dfh:98}}]\label{thm-dfh}
Let $X$ be a separable Banach space with a polyhedral norm. Then any equivalent norm on $X$ can be approximated by polyhedral norms.
\end{thm}

As is remarked at the end of \cite{ht:14}, very little is known in the non-separable case. In this paper, we will focus much of our attention on the following class of spaces.

\begin{defn}
Let $\Gamma$ be a set. The set $c_0(\Gamma)$ consists of all functions $x:\Gamma\to\R$, with the property that $\set{\gamma \in \Gamma}{|x(\gamma)| \geq \ep}$ is finite whenever $\ep > 0$. We equip $c_0(\Gamma)$ with the norm $\ndot_\infty$, where $\n{x}_\infty = \max\set{|x(\gamma)|}{\gamma \in \Gamma}$.
\end{defn}

When $\Gamma$ is uncountable, $c_0 (\Gamma)$ is non-separable. The structure of $c_0(\Gamma)$
strongly promotes the existence of the sorts of norms under discussion in this paper. For example, it is well known that the canonical norm on $c_0(\Gamma)$ is polyhedral, and that it can be approximated by $C^\infty$ smooth norms. In terms of finding positive non-separable analogues of Theorems \ref{thm-ht} and \ref{thm-dfh}, this class of spaces is a very plausible candidate.

The most general result concerning this class to date is given below. We shall call a norm $\ndot$ on $c_0 (\Gamma)$ a \emph{lattice norm} if $\n{x} \leq \n{y}$ whenever $x,y \in c_0(\Gamma)$ satisfy $|x(\gamma)| \leq |y(\gamma)|$ for each $\gamma \in \Gamma$.

\begin{thm}[{\cite[Theorem 1]{fhz:97}}]\label{thm:lattice}
Every equivalent lattice norm on $c_0(\Gamma)$ can be approximated by $C^{\infty}$ smooth norms.
\end{thm}

The following result completely settles the approximation problem in the case of $c_0(\Gamma)$, from
the point of view of $C^\infty$ smooth norms and polyhedral norms. It solves a special case of \cite[Problem 114]{hj:14}.

\begin{thm}\label{c_0-main}
Let $\Gamma$ be an arbitrary set, and let $\ndot$ be an arbitrary equivalent norm on $c_0(\Gamma)$. Then $\ndot$ can be approximated by both $C^\infty$ norms and polyhedral norms.
\end{thm}

Theorem \ref{c_0-main} is a consequence of a more general result, Theorem \ref{summble-bdry}, which involves spaces having Markushevich bases. The proofs of both results are given in Section \ref{approx-norms}. 

We turn now to the main result of Section \ref{WCG_poly_nec}. The following notion, that of a boundary of a norm, is one of the central concepts in both Sections \ref{approx-norms} and \ref{WCG_poly_nec}.

\begin{defn} Let $(X,\ndot)$ be a Banach space. A subset $B$ of the closed unit ball $B_{X^*}$ is a called a \emph{boundary} of $\ndot$ if, for each $x$ in the unit sphere $S_X$, there exists $f \in B$ such that $f(x) = 1$.
\end{defn}

This is also known as a \emph{James boundary of $X$} in the literature. The dual unit sphere $S_{X^*}$ and the set $\ext(B_{X^*})$ of extreme points of the dual unit ball $B_{X^*}$ are always boundaries of $\ndot$, by the Hahn-Banach Theorem and (the proof of the) Krein-Milman Theorem, respectively. It is worth noting that the property of being a boundary is not preserved by isomorphisms in general:\ a boundary of $\ndot$ may not be a boundary of $\tndot$, where $\tndot$ is an equivalent norm. Since we will be changing norms in this paper, it will be necessary to bear this in mind.

Boundaries play a key role in the theory of both smooth norms and polyhedral norms. If $(X,\ndot)$ has a boundary that is countable or otherwise well-behaved, then $X$ enjoys good geometric properties as a consequence -- see, for instance, \cite{b:14,f:80,fpst:14,h:95}.

Recall that an element $f \in B_{X^*}$ is called a {\em $w^*$-strongly exposed} point of $B_{X^*}$ if there exists $x \in B_X$ such that $f(x)=1$ and, moreover, $\n{f-f_n} \to 0$ whenever $(f_n)\subseteq B_{X^*}$ is a sequence satisfying $f_n(x) \to 1$. It is a simple matter to check that the (possibly empty) set $w^*$-$\strg\exp(B_{X^*})$ of $w^*$-strongly exposed points of $B_{X^*}$ is contained in any boundary of $\ndot$. We recall the following important result of Fonf, concerning polyhedral norms.

\begin{thm}[{\cite[Theorem 1.4]{f:00}}]\label{main-fonf} Let $\ndot$ be a polyhedral norm on a Banach space $X$ having density character $\kappa$. Then $w^*$-$\strg\exp(B_{X^*})$ has cardinality $\kappa$ and is a boundary of $\ndot$ (so is the minimal boundary, with respect to inclusion). Moreover, given $f \in w^*$-$\strg\exp(B_{X^*})$, the set $A_f \cap B_X$ has non-empty interior, relative to the affine hyperplane $A_f := \set{x \in X}{f(x)=1}$. 
\end{thm}

In particular, if $X$ is separable and $\ndot$ is polyhedral, then $w^*$-$\strg\exp(B_{X^*})$ is a countable boundary. Conversely, according to \cite[Theorem 3]{f:80}, if $(X,\ndot)$ is a Banach space and $\ndot$ has a countable boundary $B$, then $X$ admits equivalent polyhedral norms that approximate \ndot. Thus, in the separable case, the existence of equivalent polyhedral norms can be characterised purely in terms of the cardinality of the boundary.

In the non-separable case however, any analogous characterisations, if they exist, must generally rely on more than the cardinality of the boundary alone. There exist Banach spaces $(X,\ndot)$ having no equivalent polyhedral norms, yet $X$ has density the continuum $\mathfrak{c}$, and $\ndot$ has boundary $B$ of cardinality $\mathfrak{c}$. Such Banach spaces can take the form $X=C(T)$, where $T$ is the 1-point compactification of a suitably chosen locally compact scattered tree -- see \cite[Theorem 10]{fpst:08} for more details.

Recall that a Banach space $X$ is \emph{weakly compactly generated} (WCG) if $X = \cl{\aspan}^{\ndot}(K)$, where $K \subseteq X$ is weakly compact. Separable spaces and reflexive spaces are WCG. Examples of WCG spaces that are neither include the $c_0(\Gamma)$ spaces above. The following is our main result of Section \ref{WCG_poly_nec}. It provides a little more information about the structure of the set $w^*$-$\strg\exp(B_{X^*})$, besides cardinality, given a WCG polyhedral Banach space.

\begin{thm}\label{WCGpoly} Let $X$ be WCG, and let the norm $\ndot$ on $X$ be polyhedral. Then the boundary $w^*$-$\strg\exp(B_{X^*})$ of $\ndot$ may be written as
\[
w^*\text{-}\strg\exp(B_{X^*}) \;=\; \bigcup_{n=1}^\infty D_n,
\]
where each $D_n$ is relatively discrete in the $w^*$-topology.
\end{thm}

The theorem above should be compared to the following sufficient condition:\ if the norm $\ndot$ on $X$ admits a boundary $B$ such that $B=\bigcup_{n=1}^\infty D_n$ and $B=\bigcup_{m=1}^\infty K_m$, where each $D_n$ is relatively discrete in the $w^*$-topology, and each $K_m$ is $w^*$-compact, then $\ndot$ can be approximated by polyhedral norms \cite[Theorem 7]{fpst:14}. Thus Theorem \ref{WCGpoly} can be considered as a step towards a characterisation of the existence of polyhedral norms, in the WCG case.

\section{Approximation of norms}\label{approx-norms}

%$\ndot$ can be approximated by norms having boundaries that consist solely of elements having finite support. Consequently, by Theorem \ref{thm-b-and-fpst}, 

%\begin{defn}We say a norm $\ndot$ \emph{depends locally on finitely many coordinates} if given $x \in X$, there exists $\ep > 0$ and $f_1, \dots, f_n \in X^*$ such that for $ \n{y - x} < \ep$, $$f_i(x) = f_i(y)  \text{ for each } 1 \leq i \leq n \implies \n{x} = \n{y}.$$
%\end{defn}

%Using Theorem \ref{summble-bdry}, we are able to prove the next result, which completely settles the approximation question in the case of $c_0(\Gamma)$.

Our main results, throughout this section and the next, concern a class of spaces which include all spaces of the form $c_0(\Gamma)$, namely those that admit the following type of basis.
\begin{defn}
We call an indexed set of pairs $( e_{\gamma}, e_{\gamma}^* )_{\gamma \in \Gamma} \subseteq X \times X^*$ a \emph{Markushevich basis} (or M-basis) if 
\begin{itemize}
\item $e_\alpha^*(e_\beta) = \delta_{\alpha\beta},$ (that is, $( e_{\gamma}, e_{\gamma}^* )_{\gamma \in \Gamma}$ is a biorthogonal system);
\item $\cl{\aspan}^{\ndot}( e_{\gamma} )_{\gamma \in \Gamma} = X$, and
\item $( e_{\gamma}^* )_{\gamma \in \Gamma}$ separates the points of $X$.
\end{itemize}
Furthermore, an M-basis is called \emph{strong} if $x \in \cl{\aspan}^{\ndot}\set{e_\gamma}{e^*_\gamma(x) \neq 0}$ for all $x \in X$, \emph{shrinking} if $X^* = \cl{\aspan}^{\ndot}( e_{\gamma}^* )_{\gamma \in \Gamma}$, and \emph{weakly compact} if $\set{e_\gamma}{\gamma \in \Gamma} \cup \{0\}$ is weakly compact.
\end{defn}

The existence of an M-basis allows us to define supports of functionals in the dual space.

\begin{defn}
Let $X$ be a Banach space with an M-basis $(e_\gamma, e_{\gamma}^*)_{\gamma \in \Gamma}$ and let $f \in X^*$. Define the \emph{support} of $f$ (with respect to the basis) to be the set
\[
\supp(f) = \set{\gamma \in \Gamma}{f(e_\gamma) \neq 0}.
\]
\label{eg1-2} We say $f$ has \emph{finite support} if $\supp(f)$ is finite.
\end{defn}

The main result of this section, Theorem \ref{summble-bdry}, states that if $X$ has a strong M-basis then, given the right circumstances, the norm on $X$ can be approximated by norms having boundaries that consist solely of elements having finite support. The following result illustrates the relevance of such boundaries to the current discussion. It amalgamates two theorems, both of which are stated with broader hypotheses in their original forms.

\begin{thm}[{\cite[Theorem 2.1]{b:14} and \cite[Theorem 7]{fpst:14}}]\label{thm-b-and-fpst}
Let a Banach space $X$ have a strong M-basis, and suppose that the norm $\ndot$ has a boundary consisting solely of elements having finite support. Then $\ndot$ can be approximated by both $C^\infty$ norms and polyhedral norms.
\end{thm}

Now, for the rest of this section, we will assume that the Banach space $X$ has a strong M-basis $(e_{\gamma},e_{\gamma}^*)_{\gamma \in \Gamma}$, such that $\n{e_\gamma}=1$ for all $\gamma \in \Gamma$. Furthermore, we will suppose that there is some fixed $L\geq 0$ satisfying $\n{e^*_\gamma} \leq L$ for all $\gamma \in \Gamma$.

Given $f \in X^*$, set $\n{f}_1 = \sum_{\gamma \in \Gamma} |f(e_{\gamma})|,$ whenever this quantity is finite, and set $\n{f}_1=\infty$ otherwise. Observe that if $x = \sum_{\gamma \in F} e_{\gamma}^*(x) e_{\gamma}$, for some finite $F \subseteq \Gamma$, then
\[
|f(x)| \leq \sum_{\gamma \in F} |e_{\gamma}^*(x)|\,|f(e_{\gamma})| \leq L \n{x} \sum_{\gamma \in F} |f(e_{\gamma})| \leq L\n{x}\n{f}_1,
\]
whence $\n{f} \leq L \n{f}_1$ for all $f \in X^*$. It is also easy to see that $\ndot_1$
is a $w^*$-lower semicontinuous function on $X^*$, and that given $r>0$, the norm-bounded set
\[
W_r = \set{f \in X^*}{\n{f}_1 \leq \lambda},
\]
is symmetric, convex and $w^*$-compact.

%We shall adopt some terminology that appears in \cite[Section 3]{g:09} and is explicitly defined later in \cite{fst:14}.

%\begin{defn}[{\cite[Definition 2.3]{fst:14}}]
%We shall call $A \subseteq X^*$ \emph{summable} (with respect to $(e_{\gamma}, e_{\gamma}^*)_{\gamma \in \Gamma}$) if $\n{f}_1 < \infty$ for all $f \in A$.
%\end{defn}

Let us consider the set $B=\set{f \in S_{X^*}}{\n{f}_1 < \infty}$. Evidently, $B$ is the countable union of the sets $S_{X^*}\cap W_r$, $r\in \N$, which are $w^*$-closed in $S_{X^*}$. If $S_{X^*} \cap W_r$ contains a non-empty norm-open subset of $S_{X^*}$, for some $r \in \N$, then it is a straightforward matter to show that there exists $M\geq 0$ such that $\n{f}_1 \leq M\n{f}$ for all $f \in X^*$, whence $S_{X^*} \cap W_M=S_{X^*}$ and $X$ is isomorphic to $c_0(\Gamma)$ via the map $x \mapsto (e_\gamma^*(x))_{\gamma \in \Gamma}$. If there is no such $r$, then of course $B$ is of first category in $S_{X^*}$. If $X$ is not isomorphic to any space of the form $c_0(\Gamma)$, then $B \neq S_{X^*}$, but $B$ may still be a boundary of $\ndot$ -- see, for instance, examples in \cite{b:14,fpst:14,g:09}. We shall be interested in cases where $B$ is a boundary of $\ndot$.

The following lemma will be used in Theorem \ref{summble-bdry}.

\begin{lem}\label{shrinking}
Suppose that $B$ as defined above is a boundary of $\ndot$. Then $X^* = \cl{\aspan}^{\ndot}(e^*_\gamma)$,
i.e., the M-basis of $X$ is shrinking.
\end{lem}

\begin{proof}
Let $F \subseteq \Gamma$ be finite, and define
\[
X_F \;=\; \cl{\aspan}^{\ndot}(e_\gamma)_{\gamma \in\Gamma\setminus F}\qquad\text{and}\qquad
W_F \;=\; \aspan(e_\gamma^*)_{\gamma \in F}.
\]
Then $W_F = X_F^\perp$ (the inclusion $X_F^\perp \subseteq W_F$ follows from the fact that the basis is strong), and 
thus $X^*/W_F$ naturally identifies with $X_F^*$, and $\n{f\restrict{X_F}}=d(f,W_F)$ for all $f \in X^*$, where
\[
d(f,W_F) \;=\; \inf\set{\n{f-g}}{g \in W_F}.
\]

Suppose, for a contradiction, that there exists $f \in X^*$ and $\ep>0$, such that $d(f,W_F)>\ep$ for all finite $F \subseteq \Gamma$. Let $F_0$ be empty. Since $\n{f}=d(f,W_{F_0})>\ep$, take a unit vector $x_0 \in X$ having finite support, such that $f(x_0) > \ep$. Set $F_1=\supp x_0$. Since $\n{f\restrict{X_{F_1}}}=d(f,W_{F_1})>\ep$, there exists a unit vector $x_1 \in X$ having finite support in $\Gamma\setminus F_1$, such that $f(x_1)>\ep$. Define $F_2 = F_1 \cup \supp x_1$. Continuing like this, we get a sequence of unit vectors $(x_n)$ having finite, pairwise disjoint supports, such that $f(x_n)>\ep$ for all $i$. Clearly, $(x_n)$ is not weakly null.

On the other hand, if $f \in B$ and $y = \sum_{\gamma \in F} e^*_\gamma(y) e_\gamma$ is a unit vector,
where $F\subseteq\Gamma$ is finite, then
\[
|f(y)| \;\leq\; \sum_{\gamma\in\Gamma} |e_\gamma^*(y)|\,|f(e_\gamma)| \;\leq\; L\sum_{\gamma \in F}|f(e_\gamma)|.
\]
It follows that $f(x_n) \to 0$ as $n\to\infty$. This holds for every element of $B$, which is a boundary,
so $x_n \to 0$ weakly, by Rainwater's Theorem \cite[Theorem 3.134]{fhhspz:11}. This is a contradiction.
\end{proof}

We can now prove Theorem \ref{summble-bdry}. The method of proof owes a debt to \cite[Proposition 3.1]{g:09}, although the approximation scheme used in that result fails in the case under consideration here, and substantial modifications must be made.

\begin{thm}\label{summble-bdry}
Let a Banach space $X$ have an M-basis as above, and suppose that $B$ as above is a boundary. Given $\ep>0$, there exists an $\ep$-approximation $\tndot$ of $\ndot$, which has a boundary consisting solely of elements having finite support. Consequently, by Theorem \ref{thm-b-and-fpst}, $\ndot$ can
be approximated by $C^\infty$ smooth norms and polyhedral norms.
\end{thm}

\begin{proof}
Fix $\ep \in (0,1).$ Suppose $f \in X^*$ satisfies $\n{f}_1 < \infty$. We define a sequence of positive numbers and a sequence of subsets of $\Gamma$ inductively. To begin, set
\[
p(f,1) = \max \set{|f(e_\gamma)|}{\gamma \in \Gamma} \quad \text{and} \quad G(f,1) = \set{\gamma \in \Gamma}{|f(e_\gamma)| = p(z,1)}.
\]
Given $n \geq 2$, we define
\[
  p(f,n) = \begin{cases}
       \max \set{|f(e_\gamma)|}{\gamma \in \Gamma \backslash G(f,n-1)} & \text{if } \Gamma\backslash G(z,n-1) \neq \emptyset \\
       0 & \text{otherwise},
     \end{cases}
\]
\[
\text{and }G(f,n) = \{ \gamma \in \Gamma : |f(e_\gamma)| \geq p(f,n) \}.
\]
Observe that the set $G(f,n)$ is finite if and only if $p(f,n) \neq 0$ and, in this case, $\n{f}_1 \geq p(f,n) |G(f,n)|$. By induction, $|G(f,n)| \geq n$ for all $n$, so $p(f,n) \leq \n{f}_1 n^{-1}$ and, in particular, $p(f,n) \to 0$. By construction, the sequence $(p(f,n))$ is decreasing, and strictly decreasing on the set of indices $n$ at which it is non-zero. If $p(f,n) = 0$ for some $n \in \N,$ then $f(e_\gamma) \neq 0$ for at most finitely many $\gamma$ and hence $f$ has finite support. Thus, when $f$ has infinite support, we get a strictly decreasing sequence of positive numbers $p(f,n) \rightarrow 0$, and a strictly increasing sequence of finite sets $(G(f,n)).$

Provided $G(f,n)$ is finite, we define
\[
w(f,n) = \sum_{\gamma \in G(f,n)} \sgn(f(e_{\gamma}))e_{\gamma}^*,
\]
\[
\text{and } h(f,n) = \sum_{i=1}^n (p(f,i) - p(f, i+1))w(f,i).
\]
Let $\gamma \in \Gamma$. If $\gamma \in \Gamma \backslash \bigcup_{n=1}^{\infty} G(f,n)$, then $h(f,m)(e_\gamma) = 0 = f(e_\gamma)$ for all $m$. Otherwise, let $n$ be minimal, subject to the condition $\gamma \in G(f,n)$. By minimality, we have $p(f,n)=|f(e_\gamma)|$. If $m < n$, then $h(f,m)(e_{\gamma}) = 0$. If $m \geq n$, then we can see that
\begin{align*}
h(f,m) (e_\gamma) &= \sum_{i=n}^{m} (p(f,i) - p(f, i+1))\sgn(f(\gamma))\\
& = [p(f,n) - p(f,n+1)\\
& \quad + p(f,n+1) - p(f,n+2) \\
& \quad + \ldots - \ldots \\
& \quad + p(f,m) - p(f,m+1)]\text{sgn}(f(e_\gamma))\\
& = |f(e_\gamma)|\sgn(f(e_\gamma)) - p(f,m+1)\sgn(f(e_\gamma))\\
& = f(e_\gamma) - p(f,m+1)\sgn(f(e_\gamma)).
\end{align*}
From the calculation above and the fact that $p(f,m+1) < |f(e_\gamma)|$, we have
\[
|h(f,m) (e_\gamma)| = |\text{sgn}(f(e_\gamma)(|f(e_\gamma)| - p(f,m+1))| = |f(e_\gamma)| - p(f,m+1).
\]
Since $p(f,m+1) \geq 0$, we obtain $|h(f,m) (e_\gamma)| \leq |f(e_\gamma)|$.

Therefore, for all $\gamma \in \Gamma$, $|h(f,m) (e_\gamma)| \leq |f(e_\gamma)|$ and $h(f,m) (e_\gamma) \to f(e_\gamma)$ as $m \to \infty$. We apply Lebesgue's Dominated Convergence Theorem to conclude that $\n{f-h(f,m)}_1 \rightarrow 0.$ Since $\ndot \leq L \ndot_1, $ we also get $\n{f-h(f,m)} \rightarrow 0$. Since the signs of $w(f,i)(e_\gamma)$ and $w(f,i')(e_\gamma)$ agree whenever they are non-zero,
\[
\n{h(f,n)}_1 = \sum_{i=1}^{n} (p(f,i) - p(f, i+1))\n{w(f,i)}_1 = \sum_{i=1}^{n} (p(f,i) - p(f, i+1))|G(f,i)|.
\]
Therefore, if $f$ has infinite support, then $\n{f}_1 = \sum_{i=1}^{\infty} (p(f,i) - p(f, i+1))|G(f,i)|$.

Given $m > n$, define
\[
   g(f,n,m) = \begin{cases}
       {\displaystyle \frac{\n{f-h(f,n)}_1}{|G(f,m)|}\,w(f,m)} & \text{if } |G(f,m)| <  \infty, \\
       0 & \text{otherwise}.
     \end{cases}
\]
and $j(f,n,m) = h(f,n) + g(f,n,m)$, $m>n$. Observe that $\supp(j(f,n,m)) \subseteq G(f,m)$.

Let $B_r = B_{X^*} \cap W_r = \set{f \in B_{X^*}}{\n{f}_1 \leq r}$. Of course, $B \subseteq \bigcup_{r=1}^{\infty} B_r$. We let
\[
V_r = \set{j(f,n,m)}{f \in B_r,\, m > n \text{ and } \n{f - j(f,n,m)} < 2^{-(r+2)} \ep},
\]
\[
\text{and set }V = \bigcup_{r=1}^{\infty} (1 + 2^{-r}\ep)V_r.
\]
Define $\tn{x} = \sup \{ f(x) : f \in V \}$. This is the norm that we claim $\ep$-approximates $\ndot$ and has a boundary consisting solely of elements having finite support.

First of all, we prove that $\n{x} < \tn{x} \leq (1+\ep)\tn{x}$ whenever $x \neq 0$.
Take $x \in X$ with $\n{x} = 1$ and let $f \in B$ such that $f(x) =1$ (which is possible as $B$ is a boundary of $\ndot$). Let $r$ be minimal, such that $f \in B_r$. Since $\n{f} \leq L \n{f}_1$ for all $f \in X^*$, and $\n{f- j(f,n,m)}_1 \leq 2\n{f-h(f,n)}_1$, it follows that there exists $n$ such that $\n{f - j(f,n,m)} < 2^{-(r+2)} \ep$ whenever $m > n$. In particular,
\[
\tn{x} \geq (1+2^{-r}\ep)j(f,n,n+1)(x) \geq (1+2^{-r}\ep)(1-2^{-(r+2)}\ep) \geq 1+2^{-(r+1)}\ep.
\]

%For any such $m$,
%\begin{align*}
%1 = f(x) &= j(f,n,m)(x) + (f -j(f,n,m))(x) \\ & = (1 + 2^{-r} \ep)j(f,n,m)(x) -  2^{-r}\ep j(f,n,m)(x) + (f-j(f,n,m))(x)\\
%& \leq \tn{x} - ( 2^{-r}\ep j(f,n,m)(x) + (j(f,n,m) -f)(x)).
%\end{align*}
%
%We know $(j(f,n,m) -f)(x) > - 2^{-(r+2)} \ep$ and
%\[
%j(f,n,m)(x) \geq 1 - \n{j(f,n,m) - f}\,\n{x} \geq {\textstyle \frac{1}{2}}.
%\]
%Thus
%\[
%2^{-r}\ep j(f,n,m)(x) + (j(f,n,m) -f)(x) \geq 2^{-(r+1)} \ep - 2^{-(r+2)} \ep > 0.
%\]
To secure the other inequality, simply observe that if $f \in B_r$, $m > n$ and $\n{f - j(f,n,m)} < 2^{-(r+2)} \ep$, then
\begin{align*}
(1+2^{-r}\ep)j(f,n,m)(x) &\leq (1+2^{-r}\ep)(1 + 2^{-(r+2)} \ep)\\
&\leq 1 + (2^{-r}+2^{-(r+2)} + 2^{-(2r+2)})\ep \leq 1 + \ep.
\end{align*}
This means that $\tn{x} \leq 1+\ep$. By homogeneity, $\n{x} < \tn{x} \leq (1+\ep)\n{x}$ whenever $x \neq 0$.

Now we show that $\tndot$ has a boundary consisting solely of elements having finite support. By Milman's Theorem \cite[Theorem 3.66]{fhhspz:11}, we know that $\ext (B_{(X,\ttrin)^*}) \subseteq \cl{V}^{w^*}$. Define
\[
D =  \bigcap_{r=1}^{\infty} \left(\cl{\bigcup_{s=r}^{\infty}(1+2^{-s}\ep)V_s }^{w^*}\right),
\]
and let $d \in D$. For each $r \in \N, \n{d} \leq (1 + 2^{-r}\ep)(1+2^{-(r+2)} \ep)$, and hence $\n{d} \leq 1$. Therefore, if $\tn{x} =1$, then
\[
d(x) \leq \n{d}\,\n{x} \leq \n{x} < 1.
\]
It follows that, with respect to $\tndot$, none of the elements of $D$ are norm-attaining. Consequently, $\widetilde{B}  = \ext(B_{(X,\ttrin)^*}) \backslash D$ is a boundary of $\tndot$. We claim that every
element of $\widetilde{B}$ has finite support.

Given $f \in \widetilde{B}$, we have $f \in (1+2^{-r}\ep)\cl{V_r}^{w^*}$ for some $r \in \N$. For a contradiction, we will assume that $f$ has infinite support. According to Lemma \ref{shrinking}, our M-basis is shrinking. It follows that $\supp g$ is countable for all $g \in X^*$. Thus, $\cl{V_r}^{w^*}$ is Corson compact in the $w^*$-topology (see \cite[Definition 14.40]{fhhspz:11} for the definition), which implies that it is a Fr\'echet-Urysohn space \cite[Exercise 14.57]{fhhspz:11}. In particular, there exist sequences $(f_k) \subseteq B_r$, and $(n_k), (m_k) \subseteq \N$, with $n_k < m_k$ for all $k\in\N$, such that $(j(f_k,n_k, m_k)) \subseteq V_r$ and $j(f_k,n_k, m_k) \overset{w^*}{\longrightarrow} l$, where $l = (1+2^{-r}\ep)^{-1}f$. % will show that $f$ is a multiple of a nontrivial convex combination of elements of $V_r$.

We claim that, in fact, $f_k \overset{w^*}{\longrightarrow} l$. First, we show that $h(f_k,n_k) \overset{w^*}{\longrightarrow} l$. To this end, suppose that $|G(f_k,m_k)| \nrightarrow \infty$. Then by taking a subsequence if necessary, there exists $N \in \N$ such that $|\supp (j(f_k,n_k,m_k))| \leq |G(f_k,m_k)| \leq N$ for all $k$. But as $j(f_k,n_k,m_k) \overset{w^*}{\longrightarrow} l$, this would force $|\supp(l)| \leq N < \infty$, which is not the case. Thus we must have $|G(f_k,m_k)| \rightarrow \infty$. Therefore, for all $\gamma \in \Gamma$, $g(f_k,n_k,m_k)(e_\gamma) \to 0$ as $k\to\infty$. Since $\ndot\leq L \ndot_1,$ the sequence $(g(f_k,n_k,m_k))$ is bounded. Therefore, $g(f_k,n_k,m_k) \overset{w^*}{\longrightarrow} 0$ and hence $h(f_k,n_k) \overset{w^*}{\longrightarrow} l$.

We will now show that $f_k - h(f_k,n_k) \overset{w^*}{\longrightarrow} 0$. For each $\gamma \in \Gamma$, $|f_k(\gamma) - h(f_k,n_k)(e_\gamma)| \leq |f_k(e_\gamma)|$, so $\n{f_k - h(f_k,n_k)}_1 \leq  \n{f_k}_1.$ Therefore, $(f_k - h(f_k,n_k))$ is a bounded sequence. Given $\gamma \in \Gamma$,
\[
|(f_k-h(f_k,n_k))(e_\gamma)| \leq p(f_k,n_k+1) \leq \frac{\n{f_k}_1}{|G(f_k,n_k+1)|} \leq \frac{r}{|G(f_k,n_k+1)|}.
\]
Since $h(f_k,n_k) \overset{w^*}{\longrightarrow} l$, as above, the infinite support of $l$ ensures that
$|G(f_k,n_k)| \to \infty$. Therefore, $(f_k-h(f_k,n_k))(e_\gamma) \to 0$ and hence $f_k - h(f_k,n_k) \overset{w^*}{\longrightarrow} 0$ as $k\to \infty$. It follows that $f_k \overset{w^*}{\longrightarrow} l$ as claimed, and hence $l \in B_r$.

Fix $n \in \N$ such that $\n{l-h(l,n)}_1 < L^{-1}2^{-(r+3)}\ep$. Then for all $m > n$,
\[
\n{l-j(l,n,m)} \leq L\n{l-j(l,n,m)}_1 \leq 2L\n{l-h(l,n)}_1 < 2^{-(r+2)}\ep.
\]
So $j(l,n,m) \in V_r$ for all $m > n$. Let
\[
\lambda_m = \frac{(p(l,m) - p(l,m+1))|G(l,m)|}{\n{l-h(l,n)}_1}.
\]
Note that $\lambda_m>0$ whenever $m>n$. Since $\n{l-h(l,n)}_1 = \sum_{i=n+1}^{\infty} (p(l,i) - p(l,i+1))|G(l,i)|$, we get $\sum_{m=n+1}^{\infty} \lambda_m = 1.$
\begin{align*}
\sum_{m=n+1}^{\infty} \lambda_m j(l,n,m) & = \sum_{m=n+1}^{\infty} \lambda_m h(l,n) + \sum_{m=n+1}^{\infty} \lambda_m g(l,n,m) \\
& = h(l,n) + \sum_{m=n+1}^{\infty} (p(l,i) - p(l, i+1))w(l,i) = l .
\end{align*}
Therefore, $f$ is a nontrivial convex combination of elements of $(1+2^{-r}\ep ) V_r \subseteq B_{(X,\ttrin)^*}$, so $f \notin \ext(B_{(X,\ttrin)^*})$, and hence $f \notin \widetilde{B}$. This gives us our desired contradiction. In conclusion, $\widetilde{B}$ is a boundary of $\tndot$ consisting solely of functionals having finite support.
\end{proof}

Theorem \ref{c_0-main} becomes a trivial consequence of Theorem \ref{summble-bdry}.

\begin{proof}[Proof of Theorem \ref{c_0-main}]
In this case $B=S_{(c_0(\Gamma),\|\cdot\|)^*}$, so it is a boundary of $\ndot$.
\end{proof}

It is worth remarking that the implication (d) $\Rightarrow$ (c) of \cite[Theorem 2.5]{fst:14} is essentially Theorem \ref{summble-bdry}, but with the additional assumption that the M-basis is countable. The method of proof in that case is completely different from the one presented here.

%We conclude this section with an open problem.
%
%\begin{prob}
%Let $K$ be a scattered compact Hausdorff space, and let $\ndot$ be a norm on $C(K)$, the space of continuous real-valued functions on $K$, that is equivalent to the canonical supremum norm. Can $\ndot$ be approximated by $C^\infty$ smooth norms, or polyhedral norms? In particular, is this true when $K$ is the ordinal interval $[0,\omega_1]$, or when $K$ has empty Cantor-Bendixson derivative of order 3?
%\end{prob}

\section{A necessary condition for polyhedrality in WCG spaces}\label{WCG_poly_nec}

We begin this section with a lemma. It is based on straightforward geometry and is probably folklore, but is included for completeness since we have no direct reference for it.

\begin{lem}\label{singledependence}
Suppose that $D \subseteq B_{X^*}$ has the property that for all $f \in D$, there exists
$x_f \in X$ and $r_f >0$ such that $\n{x_f + z}=f(x_f+z)$ whenever $\n{z}<r_f$. Then
\begin{enumerate}
\item $r_f \leq \n{x_f}$, and
\item $\n{z} < r_f$ and $g \in D\setminus\{f\}$ implies $g(x_f+z)<\n{x_f+z}$.
\end{enumerate}
In particular, if $f,g \in D$ are distinct then $\n{x_g-x_f} \geq r_f$.
\end{lem}

\begin{proof}$\;$
\begin{enumerate}
\item Suppose that $\n{x_f} < r_f$. Let $y \in X$ satisfy
$\n{y} < r_f-\n{x_f}$. Then $\n{\pm y - x_f} < r_f$ and so
\[
f(y) \;=\; \n{y} \;=\; \n{-y} \;=\; f(-y) \;=\; -f(y),
\]
meaning that $y \in \ker f$. It follows that $f=0$, which is impossible.
\item Suppose $\n{z}<r_f$, $g \in D\setminus\{f\}$ and $g(x_f + z)=\n{x_f + z}$.
Since $g \neq f$ we can find $y \in \ker f$ such that $g(y)>0$ and $\n{y}<r_f-\n{z}$.
Otherwise we would have $\ker f \subseteq \ker g$, so $g = \alpha f$ for some $\alpha$,
and since $f(x_f+z)=\n{x_f+z}=g(x_f+z)=\alpha f(x_f+z)$, and $\n{x_f+z} > 0$ by (1),
we conclude that $g=f$, which is not the case. Thus $\n{y+z} < r_f$ and so
\[
\n{x_f + y + z} \;=\; f(x_f + y + z) \;=\; f(x_f + z).
\]
On the other hand,
\[
\n{x_f + y + z} \;\geq\; g(x_f + y + z) \;>\; g(x_f + z) \;=\; \n{x_f+z} \;=\; f(x_f+z).
\]
\end{enumerate}
Finally, if $f,g \in D$ are distinct and $\n{x_g-x_f} < r_f$, then by (2) we would have
\[
\n{x_g} \;=\; g(x_g) \;=\; g(x_f + (x_g-x_f)) \;<\; \n{x_f + (x_g-x_f)} \;=\; \n{x_g}. \qedhere
\]
\end{proof}

Armed with this lemma, we can give the proof of Theorem \ref{WCGpoly}.

\begin{proof}[Proof of Theorem \ref{WCGpoly}] Since $X$ is WCG, we can find a weakly compact M-basis $(e_\gamma,e^*_\gamma)_{\gamma \in \Gamma}$ of $X$ (see, for instance \cite[Theorem 13.16]{{fhhspz:11}}). Let $E_n$ be the set of $x \in X$ that can be written as a linear combination of at most $n$ elements of $(e_\gamma)_{\gamma \in \Gamma}$. Let us define $B:=w^*$-$\strg\exp(B_{X^*})$. According to Theorem \ref{main-fonf}, for each $f \in B$, we can find a point $x \in \aspan(e_\gamma)_{\gamma \in \Gamma}$ that lies in the interior of $A_f \cap B_X$, where $A_f$ is the supporting hyperplane as defined in that theorem. By a straightforward argument, it follows that there exists $r>0$ such that $\n{x + z}=f(x+z)$ whenever $\n{z}<r$. Any such $x$ belongs to some $E_n$. Therefore, given $f \in B$, we can define $n_f$ to be the minimal $n\in\N$ for which we can find an $x$ and $r$ as above, with $x \in E_n$.

Define $D_{n,m}$ to be the set of all $f \in B$ such that $n_f=n$, and there exist $x$ and $r$, as described above, which in addition satisfy $r \geq 2^{-m}$ and
\[
x = \sum_{\gamma \in F} a_\gamma e_{\gamma},
\]
where $F\subseteq \Gamma$ has cardinality $n$ and $|a_\gamma| \leq m$ for all $\gamma \in F$. Any such pair $(x,r)$ will be called a witness for $f \in D_{n,m}$.

Evidently, $B=\bigcup_{n,m=1}^\infty D_{n,m}$. We claim that each $D_{n,m}$ is relatively discrete in the norm topology. For a contradiction, suppose otherwise and let $f,\, f_k \in D_{n,m}$ such that $\n{f-f_k}\to 0$. For each $k \in \N$, select a witness $(x_k,r_k)$ for $f_k$. The set
\[
L = \set{\sum_{\gamma \in F} a_\gamma e_{\gamma}}{F \subseteq \Gamma \text{ has cardinality $n$ and }|a_\gamma| \leq m \text{ for all }\gamma \in F},
\]
is weakly compact, being a natural continuous image of $[-m,m]^n \times (\set{e_\gamma}{\gamma \in \Gamma} \cup \{0\})^n$. Thus, by the Eberlein-\v Smulyan Theorem, and by taking a subsequence of $(x_k)$ if necessary, we can assume that the $x_k$ tend weakly to some $y \in L$. We claim that $y \in E_j$ for some $j < n$. Indeed, if
\[
y \;=\; \sum_{\gamma \in F} a_\gamma e_{\gamma},
\]
where $F\subseteq\Gamma$ has cardinality $n$ and $a_\gamma \neq 0$ for all $\gamma\in F$, then there exists a $K$ for which $e^*_{\gamma}(x_k) \neq 0$ for all $\gamma\in F$ and all $k \geq K$. Because each $x_k$ can be expressed as a linear combination of $n$ elements of $(e_\gamma)_{\gamma \in \Gamma}$, it follows that $x_k \in \aspan(e_{\gamma})_{\gamma\in F}$ whenever $k \geq K$. Indeed, if
\[
w \;=\; \sum_{\gamma \in G} b_\gamma e_\gamma,
\]
where $G \subseteq \Gamma$ has cardinality $n$, and if $e^*_\gamma(w) \neq 0$ for all $\gamma\in F$, then necessarily $F \subseteq G$, and equality of these sets follows since their cardinalities agree. Because the $x_k$, $k \geq K$, belong to a finite-dimensional space, it follows that $\n{y-x_k} \to 0$. However, by Lemma \ref{singledependence}, we know that the $x_k$ are uniformly separated in norm by $2^{-m}$ $(\leq r_k)$, so they cannot converge in norm to anything.

Thus $y \in E_j$ for some $j < n$, as claimed. Now fix $z \in X$ such that $\n{z} < 2^{-m}$.
We have $\n{x_k + z} = f_k(x_k+z)$ for all $k$, because $2^{-m} \leq r_k$. As $\n{f-f_k} \to 0$ and $x_k+z \to y + z$ weakly,
we get $\n{x_k+z} \to f(y+z) \leq \n{y+z}$. On the other hand, by $w$-lower semicontinuity of
the norm, $\n{y+z} \leq f(y+z)$. So the equality $\n{y+z}=f(y+z)$ holds whenever $\n{z} < 2^{-m}$. In particular, $1 = \n{x_k} \to \n{y}$. However $y \in E_j$ and $j<n$, and this contradicts the minimal choice of $n_f=n$.

Thus each $D_{n,m}$ is relatively discrete in the norm topology. Since $D_{n,m} \subseteq B$ and since the norm and $w^*$-topologies agree on $B$, it follows that $D_{n,m}$ is relatively discrete in the $w^*$-topology as well. 
\end{proof}

Finally, we recall that a Banach space $X$ is called {\em weakly Lindel\"of determined} (WLD) if $B_{X^*}$ is Corson compact in the $w^*$-topology. The class of WLD spaces includes all WCG spaces. Any polyhedral Banach space is an Asplund space (this follows, for example, from \cite[Proposition 3.7]{f:00}), and any WLD Asplund space is WCG \cite[Theorem 8.3.3]{fa:97}. Therefore Theorem \ref{WCGpoly} extends to all WLD polyhedral spaces.

%If a set $D$ is norm discrete, then $D = \bigcup_{n=1}^\infty H_n$, where
%\[
%H_n \;=\; \set{x \in D}{\n{x -y} \geq 2^{-n} \text{ whenever }y \in D\setminus\{x\}}.
%\]
%Thus we can split any discrete set into a countable union of sets that are uniformly separated in norm. Therefore we can decompose $B$ above into such a union of sets. Moreover, since the norm and $w^*$-topologies agree on $B$, any norm discrete subset of $B$ is also $w^*$-discrete (even $w^*$-`slicely'-discrete).

\medskip\emph{Acknowledgment.} The authors would like to thank P.\ H\'ajek and S.\ Troyanski for bringing the topic of Section \ref{approx-norms} to their attention, and for useful discussions concerning the matter.

\end{document}